\documentclass[a4paper,12pt]{article}

\makeatletter
\@addtoreset{footnote}{page}
\makeatother

\usepackage{style-enumitem}

\usepackage{amsthm}
\usepackage{amsmath,amssymb,latexsym,amsfonts,mathrsfs}

\renewcommand{\Im}{\mathop{\rm Im}}

\newcommand{\Supp}{\mathop{\rm Supp}}

\newcommand{\eps}{\ensuremath{\varepsilon}}

\renewcommand{\hat}{\widehat}
\renewcommand{\tilde}{\widetilde}

\newcommand{\bC}{\ensuremath{\mathbb{C}}}

\newcommand{\bE}{\ensuremath{\mathbb{E}}}

\newcommand{\bN}{\ensuremath{\mathbb{N}}}

\newcommand{\bP}{\ensuremath{\mathbb{P}}}

\newcommand{\bR}{\ensuremath{\mathbb{R}}}

\newcommand{\cL}{\ensuremath{\mathcal{L}}}
\newcommand{\cM}{\ensuremath{\mathcal{M}}}

\newcommand{\cP}{\ensuremath{\mathcal{P}}}

\theoremstyle{plain}
\newtheorem{Thm}{Theorem}[section]

\newtheorem{Prop}[Thm]{Proposition}
\newtheorem{Cor}[Thm]{Corollary}

\theoremstyle{definition}

\newtheorem{Rem}[Thm]{Remark}

\numberwithin{equation}{section}

\makeatletter
\renewcommand\section{\@startsection {section}{1}{\z@}%
                                   {-3.5ex \@plus -1ex \@minus -.2ex}%
                                   {2.3ex \@plus.2ex}%
                                   {\normalfont\large\bf}}
\makeatother

\makeatletter
\renewcommand\subsection{\@startsection {subsection}{1}{\z@}%
                                   {-3.5ex \@plus -1ex \@minus -.2ex}%
                                   {2.3ex \@plus.2ex}%
                                   {\normalfont\normalsize\bf}}
\makeatother

\usepackage{color}

\begin{document}
	
	\title{\textbf{Reproduction of initial distributions from the first hitting time distribution for birth-and-death processes}\footnote{
	This research was supported by RIMS and by ISM.}}

	\author{Kosuke Yamato\footnote{Graduate School of Science, Kyoto
	University, Japan.}
	\footnote{The research of this author was supported by
	JSPS Open Partnership Joint Research Projects grant no. JPJSBP120209921.}
	\footnote{The research of this author was supported by JSPS KAKENHI Grant Number JP21J11000}
	\quad and \quad Kouji Yano\footnotemark[2] \footnotemark[4]
	\footnote{The research of this author was supported by
	JSPS KAKENHI grant no.'s JP19H01791 and JP19K21834.}}

	\date{}

	\maketitle

	\begin{abstract}
		For birth-and-death processes, we show that every initial distribution is reproduced from the first hitting time distribution.
		The reproduction is done by applying to the distribution function a differential operator defined through the eigenfunction of the generator.
		Using the spectral theory for generalized second-order differential operators, we study asymmetric random walks.
	\end{abstract}

\section{Introduction}






Birth-and-death processes are continuous-time Markov chains on $S = \{ 0,1,2,\cdots \}$ which only jump to neighboring points.
Let us consider a birth-and-death process $X$ with positive birth and death rates on $S \setminus \{ 0 \}$ and stopped at $0$. Let $\nu$ be its initial distribution and $T_0$ be its first hitting time of $0$.
Our aim is to reproduce the initial distribution $\nu$ from the first hitting time distribution $\bP_{\nu}[T_0 \in dt]$,
where $\bP_{\nu}$ denotes the underlying probability measure of $X$ under the initial distribution $\nu$.
We show that the reproduction of the initial distribution can be done via a differential operator obtained from the eigenfunctions of the generator.

A key tool is the spectral theory.
As we will see, there always exist the first hitting time densities $f_i(t) \ (i \geq 1, t > 0)$, that is, there are functions such that $\bP_{i}[T_0 \in dt] = f_{i}(t)dt$. In addition, the density $f_i$ has a spectral representation:
\begin{align}
	f_i(t) = \pi_{i}\int_{0}^{\infty}\mathrm{e}^{-\theta t} \psi_{-\theta }(i)\rho(d\theta), \label{eq130}
\end{align}
where $\pi$ is the speed measure, $\rho$ is the spectral measure and $\psi_{-\theta}$ is a Dirichlet $(-\theta)$-eigenfunction of the generator, whose precise definitions will be given in Section \ref{section:BD}.
As is well-known (e.g., Karlin and McGregor \cite{KarlinMcGregor}), the transition probability has the spectral representation:
\begin{align}
	\bP_{i}[X_t = j] = \pi_{i}\int_{0}^{\infty}\mathrm{e}^{-\theta t} \psi_{-\theta}(i)\psi_{-\theta}(j)\rho(d\theta). \label{eq131}
\end{align} 
Comparing the representations in \eqref{eq130} and \eqref{eq131} and changing the order of the integration and the differentiation formally, we see the following result.
We will introduce a matrix $C$ in Proposition \ref{matrix_C} whose columns are generalized $0$-eigenfunctions of the generator $Q$ in the sense that $C_0 = 0$ and $C_{j-1} = QC_{j}$ for $j \geq 1$.
We write $\partial_t = \frac{d}{dt}$ and define $\psi_{\partial_t}(j)$ by the differential polynomial
\begin{align}
	\psi_{\partial_t}(j) = \sum_{k = 1}^{j}C(j,k)\partial_t^{k-1} \quad (j \geq 1). \label{eq111}
\end{align}

\begin{Thm} \label{rep_transition}
	Let $\nu$ be a probability measure on $S \setminus \{0\}$. Then it holds
	\begin{align}
		\bP_{\nu}[X_t = j] = \psi_{\partial_t}(j)f_{\nu}(t) \quad (j \geq 1), \label{eq113}
	\end{align}
	where 
	\begin{align}
		\bP_{\nu}[X_t = j] := \sum_{i = 1}^{\infty}\nu\{ i \} \bP_{i}[X_t = j]
		\quad \text{and} \quad
		f_{\nu}(t) := \sum_{i = 1}^{\infty} \nu\{i\} f_i(t). \label{}
	\end{align}

\end{Thm}
By taking $t \to 0$ in \eqref{eq113}, we obtain the reproduction of the initial distribution:
\begin{Cor} \label{rep_init}
	Let $\nu$ be a probability measure on $S \setminus \{0\}$. Then it holds
	\begin{align}
		\nu\{j\} = \lim_{t \to + 0}\psi_{\partial_t}(j)f_\nu(t) \quad (j \geq 1). \label{eq114}
	\end{align}
\end{Cor}

To obtain the reproduction formula in the explicit form, we need to compute the spectral measure of the generator.
For this purpose, we look at birth-and-death processes as {\it generalized diffusions},
and apply the spectral theory for the generalized second-order ordinary differential operators.
Using Doob's $h$-transform, we compute the matrix $C$ for asymmetric random walks.

In appendix, we study reproduction of the initial distributions for one-dimensional diffusions. We will show that under the existence of the Laplace transform of the spectral measure reproduction is possible for initial distributions with square integrable densities with respect to the speed measure.
The main tool is the spectral theory, especially the generalized Fourier transform and the spectral representation of the first hitting time densities.


\subsection*{Background of the study}

In \cite{preprintQSD}, we have studied quasi-stationary distributions of one-dimensional diffusions.
Let $X$ be a $\frac{d}{dm}\frac{d}{ds}$-diffusion on $S := [0,b) \ (0 < b \leq \infty)$ stopped at $0$.
Let us denote the set of probability distributions on $S$ by $\cP(S)$ or $\cP S$.
For a set of initial distributions $\cP \subset \cP(S \setminus \{0\})$, we say that the {\it first hitting uniqueness} holds on $\cP$ if the following holds:
\begin{align}
	\cP \ni \mu \mapsto \bP_\mu[T_0 \in dt] \quad \text{is injective}, \label{eq133}
\end{align}
where $T_x$ denotes the first hitting time of $x$ for $x \in S$. 
Recall a probability distribution $\nu$ is called a {\it quasi-stationary distribution} of $X$ when the following holds:
\begin{align}
	\bP_{\nu}[X_t \in dx \mid T_0 > t] = \nu(dx) \quad (t > 0). \label{}
\end{align}
Define
\begin{align}
	\cP_{\mathrm{exp}} := \{ \mu \in \cP(0,b) \mid \bP_\mu[T_0 \in dt] = \lambda \mathrm{e}^{-\lambda t}dt \quad \text{for some } \lambda > 0 \}.
\end{align}
One of the main results in \cite{preprintQSD} was the following:
\begin{Thm}[{\cite[Theorem 1.1]{preprintQSD}}]\label{main-theorem-03}
	Let $X$ be a $\frac{d}{dm}\frac{d}{ds}$-diffusion on $[0,b) \ (0 < b  \leq \infty)$
	stopped at $0$ satisfying
	\begin{align}
		\bP_x[T_0 < \infty] = 1 
		\quad \text{and} \quad
		\bP_x[T_y < \infty] > 0 \quad (x \in (0,b), y \in [0,b)). \label{}
	\end{align}
	Set
	\begin{align}
		\mu_{t}(dx) := \bP_{\mu}[X_t \in dx \mid T_0 > t]. \label{}
	\end{align}
	Assume the first hitting uniqueness holds on $\cP_{\mathrm{exp}}$ and
	\begin{align}
		{\bP_{\nu}}[T_0 \in dt] = \lambda \mathrm{e}^{-\lambda t}dt \quad \text{for some}
		\ \lambda > 0 \ \text{and some}\  {\nu} \in \cP(0,b).
	\end{align}
	Then for $\mu \in \cP(0,b)$ and $\lambda > 0$, the following are equivalent:
	\begin{enumerate}
		\item $
		\lim_{t \to \infty}\frac{\bP_{\mu}[T_0 > t + s]}{\bP_{\mu}[T_0 > t]} = \mathrm{e}^{-\lambda s} \ (s > 0)$.
		\item $\bP_{\mu_{t}}[T_0 \in ds] \xrightarrow[t \to \infty]{} \lambda \mathrm{e}^{-\lambda s}ds$.
		\item $\mu_t \xrightarrow
		[t \to \infty]{} {\nu}$.
	\end{enumerate}
\end{Thm}
Here the convergence of probability distributions is in the sense of the weak convergence.
From Theorem \ref{main-theorem-03}, we can see that if there is a quasi-stationary distribution $\nu_\lambda$ with $\bP_{\nu_\lambda}[T_0 > t] = \mathrm{e}^{-\lambda t}$, the convergence $\mu_{t} \xrightarrow[t \to \infty]{w}\nu_\lambda$ is reduced to the tail behavior of $T_0$.

Rogers \cite{RogersFPP} has studied the first hitting uniqueness on $\cP(0,b)$ for one-dimensional diffusions and gave a sufficient condition for it by a condition on the resolvent density. 
His condition was, however, too strong; he gave in the paper an example of a diffusion satisfying the first hitting uniqueness but not his condition. 

Reproduction of initial distributions for a set $\cP \subset \cP(S)$ obviously implies the first hitting uniqueness on $\cP$. More precisely, we can see that the reproduction formula \eqref{eq114} provides the inverse of the map \eqref{eq133}.

\subsection*{Outline of the paper}
	The remainder of the present paper is organized as follows.
	In Section \ref{section:BD}, we will recall some basic notion and setup notation for birth-and-death processes.
	In Section \ref{section:0-eigenfunc}, we will explain the scale function and speed measure for birth-and-death processes and construct the matrix $C$.
	In Section \ref{section:theta_eigenfunc}, we will show the equality \eqref{eq111} and see that the first hitting time densities have the spectral representations.
	In Section \ref{section:proof_rep_formula}, we will prove Theorem \ref{rep_transition}.
	In Section \ref{section:GeneralizedDiff}, we will see birth-and-death processes as generalized diffusions and apply the spectral theory and $h$-transforms for the processes.
	In Section \ref{section:exBD}, we will discuss symmetric and asymmetric random walks.
	In Appendix \ref{section:1-dimDiff}, we will study the reproduction of the initial distributions for one-dimensional diffusions. We will also give examples.

\section{Birth-and-death processes} \label{section:BD}

Let us consider a {\it birth-and-death process} on $S = \{ 0,1,2,\cdots \}$ such that $0$ is a trap, which is a continuous-time Markov chain with neighboring jumps such that $0$ is a trap.
We denote the birth rates by $\{\lambda_i\}_{i \geq 1}$ and the death rates by $\{ \mu_i \}_{i \geq 1}$:
\begin{align}
	\bP_{i}[X_{\tau} = i+1] = \frac{\lambda_{i}}{\lambda_{i}+ \mu_{i}}, \quad
	\bP_{i}[X_{\tau} = i-1] = \frac{\mu_{i}}{\lambda_{i}+ \mu_{i}}
	\quad \text{and} \quad
	\bP_{i}[\tau > t] = \mathrm{e}^{-(\lambda_{i} + \mu_{i})}
	\label{}
\end{align}
where $\tau = \inf \{ t > 0 \mid X_{\tau} \neq X_0 \}$ denotes the first exit time from the initial state.
We set $\mu_0 = \lambda_0 = 0$ according to the assumption that $0$ is a trap.
We assume the birth and death rates are positive on $S \setminus \{ 0 \}$: $\lambda_{i} > 0,\ \mu_{i} > 0 \ (i \geq 1)$. 
Its $Q$-matrix 
\begin{align}
	Q = (Q(i,j))_{i,j \geq 0} = 
	\begin{pmatrix}
		Q(0,0) & Q(0,1) & \cdots \\
		Q(1,0) & Q(1,1) & \cdots \\
		\vdots & \vdots & \ddots 
	\end{pmatrix}
\end{align}
is given by
\begin{align}
	Q(i,i-1) = \mu_i \quad (i \geq 1) \quad \text{and} \quad Q(i,i+1) = \lambda_i, \quad Q(i,i) = -(\lambda_i + \mu_i) \quad (i \geq 0). \label{q-matrix}
\end{align}
We denote the transition probability by $P_t$, that is,
\begin{align}
	\bP_i[X_t = j] = P_t(i,j) \quad (t \geq 0,\ i,j \geq 0). \label{}
\end{align}
For this process, the transition probability $P_t(i,j)$ satisfies the forward and backward equations: $\partial_t P_t =  P_t Q = Q P_t$, that is, 
\begin{align}
	\partial_t P_t(i,j) &= \lambda_{j-1}P_t(i,j-1) - (\lambda_j + \mu_j)P_t(i,j) + \mu_{j+1} P_t(i,j+1) \quad (i,j \geq 0), \label{F-eq} \\
	\partial_t P_t(i,j) &= \mu_{i}P_t(i-1,j) - (\lambda_i + \mu_i)P_t(i,j) + \lambda_{i} P_t(i+1,j) \quad (i,j \geq 0), \label{B-eq}
\end{align}
where we understand $\lambda_{-1}P_t(i,-1) = 0$ (see e.g., Anderson \cite[Chapter 2, Theorem 2.2]{Anderson}).

Define the {\it speed measure} $\pi = (\pi_i)_{i \in S \setminus \{0\}}$ on $S \setminus \{0\}$ by
\begin{align}
	\pi_1 = 1, \quad \pi_{i} = \frac{\lambda_1 \cdots \lambda_{i-1}}{\mu_2 \cdots \mu_{i}} \quad (i \geq 2). \label{speed-measure}
\end{align}
Note that $\pi$ is characterized with $\pi_1 = 1$ and the following balancing condition:
\begin{align}
	\pi_{i+1} \mu_{i+1} = \pi_{i} \lambda_i \quad (i \geq 1). \label{balancing_condition}
\end{align} 
Define the operator $Q$ by 
\begin{align}
	Qf(i) = \sum_{j = 1}^{\infty}Q(i,j)f(j) = \mu_i f(i-1) - (\lambda_i + \mu_i)f(i) + \lambda_i f(i+1) \quad (i \geq 0) \label{Q-operator}
\end{align}
for a function $f: S \to \bR$, where we understand $\mu_0 f(-1) = 0$.
Then $Q$ is symmetric w.r.t.\ $\pi$ in the sense that for functions $f,g: S \to \bR$ which are finitely supported and satisfy $f(0) = g(0) = 0$ the following holds:
\begin{align}
	\sum_{i = 1}^{\infty}(Qf)(i)g(i)\pi_{i} = \sum_{i = 1}^{\infty}f(i)(Qg)(i)\pi_{i}. \label{}
\end{align} 
Indeed, by \eqref{balancing_condition} it holds
\begin{align}
	&\sum_{i = 1}^{\infty}(Qf)(i)g(i)\pi_{i} \label{} \\
	= &\sum_{i = 1}^{\infty}(\mu_i f(i-1) - (\lambda_i + \mu_i)f(i) + \lambda_i f(i+1))g(i)\pi_{i} \label{} \\
	 = &\sum_{i = 1}^{\infty}\mu_i f(i-1)g(i)\pi_{i} - \sum_{i = 1}^{\infty}(\lambda_i + \mu_i)f(i)g(i)\pi_{i} + \sum_{i = 1}^{\infty}\lambda_i f(i+1)g(i)\pi_{i} \label{} \\
	 = &\sum_{i = 1}^{\infty}\lambda_i f(i)g(i+1)\pi_{i} - \sum_{i = 1}^{\infty}(\lambda_i + \mu_i)f(i)g(i)\pi_{i} + \sum_{i = 1}^{\infty}\mu_i f(i)g(i - 1)\pi_{i} \label{} \\
	 = &\sum_{i = 1}^{\infty}f(i)(Qg)(i)\pi_{i}. \label{}
\end{align}

\section{Generalized $0$-eigenfunctions} \label{section:0-eigenfunc}

Let us recall the spectral representation of the transition probability $P_t$ studied in Karlin and McGregor \cite{KarlinMcGregor}.
We introduce the scale function of $Q$:
\begin{Prop} \label{prop:scale_function}
	The scale function of $Q$ defined by
	\begin{align}
		s(i) := \left\{
		\begin{aligned}
			&0 & (i = 0), \\
			&\frac{1}{\mu_1} + \sum_{j = 1}^{i-1}\frac{1}{\pi_{j}\lambda_{j}} & (i \geq 1)
		\end{aligned}
		\right. \label{scale-function}
	\end{align}
	satisfies $Qu = 0$ and $u(0) = 0$.
	Conversely, if $u$ is a function satisfying $Qu = 0$ and $u(0) = 0$, then $u$ is a constant multiple of $s$.
\end{Prop}

\begin{proof}
	From the definition of $s$, it holds
	\begin{align}
		Qs(i) &= \mu_i s(i-1) - (\lambda_i + \mu_i)s(i) + \lambda_i s(i+1)  \label{} \\
		&= -\frac{\mu_{i}}{\pi_{i-1}\lambda_{i-1}} + \frac{\lambda_{i}}{\pi_{i}\lambda_{i}} = 0. \label{}
	\end{align}
	
	Let $u$ be another function satisfying $u(0) = 0$ and
	\begin{align}
		Qu(i) = \mu_i u(i-1) - (\lambda_i + \mu_i)u(i) + \lambda_i u(i+1) = 0 \quad (i \geq 1). \label{eq125}
	\end{align}
	Set $c = u(1)$. If $c = 0$, from \eqref{eq125}, it inductively holds $u(i) = 0 \ (i \geq 0)$. 
	If $c \neq 0$, set $v = s - (1/c)u$. Then $v$ satisfies $Qv = 0$ and $v(0) = v(1) = 0$.
	Thus it follows that $v(i) = 0 \ (i \geq 0)$, and we obtain $u= cs$. 
\end{proof}

We introduce two difference operators $D_\pi$ and $D_s$ whose composition is the $Q$-matrix.
\begin{Prop}
Define
\begin{align}
	D_{\pi} f(i) := \frac{f(i) - f(i-1)}{\pi_{i}} \quad(i \geq 1)\quad \text{and} \quad D_s f(i) := \frac{f(i+1) - f(i)}{s(i+1) - s(i)} \quad (i \geq 0). \label{dmds} 
\end{align}
Then it holds $Qf(i) = D_{\pi} D_sf(i) \ (i\geq 1)$ for every positive measurable function $f : S \to \bR$.
\end{Prop}

\begin{proof}
It holds from \eqref{balancing_condition}
\begin{align}
	D_{\pi} D_sf(i) &= D_{\pi}(\pi_{i}\lambda_i(f(i+1) - f(i)) \label{} \\
	&= \frac{\pi_{i}\lambda_i(f(i+1) - f(i)) - \pi_{i-1}\lambda_{i-1}(f(i) - f(i-1))}{\pi_{i}} \label{} \\
	&= \lambda_i(f(i+1) - f(i)) - \mu_{i}(f(i) - f(i-1)) \label{} \\
	&= Qf(i). \label{}
\end{align}
\end{proof}

We introduce a matrix $C$ which we will use in Proposition \ref{eigen} to represent an eigenfunction of $Q$.

\begin{Prop} \label{matrix_C}
	There exists a unique matrix 
	\begin{align}
		C = (C(i,j))_{i,j \geq 0} = (C_0, C_1, \cdots) 
		\quad \text{with} \quad 
		C_j = (C(i,j))_{i \geq 0} = 
		\begin{pmatrix}
			C(0,j) \\
			C(1,j) \\
			\vdots
		\end{pmatrix}
		\quad (j \geq 0)
		\label{matrixC}
	\end{align}
	which satisfies the relation
	\begin{align}
		Q C_{j} = C_{j-1} \quad (j \geq 1) \label{rec-rel}
	\end{align}
	with 
	\begin{align}
		C_0 = 0,\quad C_1(i) = s(i) \quad (i \geq 1) \quad \text{and} \quad  C(i,j) = 0 \quad  (i < j), \label{}
	\end{align}
	where $s$ is the scale function given in \eqref{scale-function}.
\end{Prop}

\begin{Rem}
	Note that $C_j$ is a generalized $0$-eigenfunction of rank $j$ in the sense that
	$Q^{j}C_{j} = 0$ and $Q^{j-1}C_{j} = s \neq 0$.
\end{Rem}

\begin{proof}[Proof of Proposition \ref{matrix_C}]
	The matrix $C$ must satisfy
	\begin{align}
		C(i,j-1) &= \sum_{k=0}^{\infty}Q(i,k)C(k,j) \label{} \\
		&= \mu_i C(i-1,j) - (\lambda_i + \mu_i) C(i,j) + \lambda_i C(i+1,j) \quad (i \geq 1, j \geq 0). \label{}
	\end{align}
	Let us determine $C_j$ recursively.
	For $j = 2,3,4,\cdots$, being fixed, suppose $C_{j-1}$ has been determined.
	For $i \geq j - 1$, we have
	\begin{align}
		C(i+1,j) = \frac{1}{\lambda_i} 
		(C(i,j-1) -  \mu_i C(i-1,j) + (\lambda_i + \mu_i) C(i,j)). \label{eq120} 
	\end{align}
	Substituting $i=j-1, j$ in \eqref{eq120}, we see that $C(j,j)$ and $C(j+1,j)$ are determined as
	\begin{align}
		C(j,j) &= \frac{C(j-1,j-1)}{\lambda_{j-1}}, \label{eq110} \\
		C(j+1,j) &= \frac{1}{\lambda_j} 
		(C(j,j-1) + (\lambda_j + \mu_i) C(j,j)), \label{eq122} 
	\end{align}
	respectively. 
	By \eqref{eq120}, the entries $C(i,j)$ for $i \geq j$ are determined recursively.
\end{proof}

\section{$\theta$-eigenfunctions and spectral representations} \label{section:theta_eigenfunc}

The eigenfunctions for $Q$ can be obtained as the generating functions corresponding to the matrix $C$.
\begin{Prop} \label{eigen}
For $\theta \in \bR$, define
\begin{align}
	\psi_{\theta}(i) = \sum_{j=1}^{\infty}C(i,j)\theta^{j-1} =  \sum_{j=1}^{i}C(i,j)\theta^{j-1}. \label{eq103}
\end{align}
Then $u = \psi_{\theta}$ is the unique solution of the following equation:
\begin{align}
	Qu = \theta u \quad \text{with} \ u(0) = 0, \ D_s u(0) = 1. \label{eq107}
\end{align}
\end{Prop}

\begin{proof}
	It is obvious that $\psi_{\theta}(0) = 0$.
	It holds
	\begin{align}
		D_s\psi_{\theta}(0) = \mu_{1}(\psi_{\theta}(1) - \psi_{\theta}(0)) = \mu_1 C(1,1) = 1. \label{}
	\end{align}
	Note that for
	\begin{align}
		\psi_\theta = (\psi_{\theta}(i))_{i \geq 0} =  
		\begin{pmatrix}
			\psi_{\theta}(0) \\
			\psi_{\theta}(1) \\
			\vdots
		\end{pmatrix}
		\quad \text{and} \quad
		r_{\theta} := (r_{\theta}(i))_{i \geq 0} =  
		\begin{pmatrix}
			0 \\
			1 \\
			\theta \\
			\theta^2 \\
			\vdots
		\end{pmatrix}, \label{}
	\end{align}
	it holds
	\begin{align}
		\psi_\theta = C r_{\theta} = \sum_{j=1}^{\infty}C_j\theta^{j-1} \label{}. \label{}
	\end{align}
	Thus from \eqref{rec-rel} we have
	\begin{align}
		Q\psi_{\theta} = \sum_{j=1}^{\infty}QC_{j} \theta^{j-1} =  \sum_{j=1}^{\infty}C_{j-1}\theta^{j-1} = \theta\sum_{j=1}^{\infty}C_j\theta^{j-1} = \theta \psi_{\theta}, \label{}
	\end{align}
	where we note that $C_0 = 0$.
	Hence the function $\psi_{\theta}$ is the solution of \eqref{eq107}. 

	We check the uniqueness. Suppose $u = f$ satisfies the equation \eqref{eq107}.
	Then $v := f - \psi_{\theta}$ satisfies $Qv = \theta v$ with $v(0) = 0$ and $D_sv(0) = 0$, which implies $v(1) = 0$.
	It follows from the definition of $Q$ that $v(i) = 0 \ (i \geq 2)$.
	Thus it follows $f = \psi_{\theta}$.
\end{proof}

From the spectral theory (see e.g., Karlin and McGregor \cite[p.501]{KarlinMcGregor}), there exists a measure $\rho$ on $[0,\infty)$, which we call the spectral measure of $Q$, such that 
\begin{align}
	\sum_{i = 1}^{\infty}f(i)^2\pi_i = \int_{0}^{\infty}\hat{f}(\theta)^2\rho(d\theta), \label{}
\end{align}
for every finitely supported function $f$ with $f(0) = 0$,
where $\hat{f}$ is the {\it generalized Fourier transform} of $f$:
\begin{align}
	\hat{f}(\theta) = \sum_{i = 1}^{\infty}f(i)\psi_{-\theta}(i)\pi_i. \label{}
\end{align}
The map $f \mapsto \hat{f}$ can be extended to the unitary map between $L^2(\pi)$ and $L^2(\rho)$ and the functions $\{ \psi_{-\theta}(i)\}_{i \geq 1}$ comprise an orthogonal basis of $L^2(\rho)$ and satisfies
\begin{align}
	\frac{\delta_{ij}}{\pi_{j}} = \int_{0}^{\infty}\psi_{-\theta}(i)\psi_{-\theta}(j)\rho(d\theta) \quad (i,j \geq 1), \label{Orthogonal}
\end{align}
where $\delta_{ij}$ is the Kronecker delta.
We now have the spectral representation of the transition probability $P_t$:
\begin{align}
	\frac{1}{\pi_{j}}P_t(i,j) = \int_{0}^{\infty}\mathrm{e}^{-\theta t}\psi_{-\theta}(i)\psi_{-\theta}(j)\rho(d\theta) \quad (i,j \geq 1, t \geq 0). \label{BD-density}
\end{align}


From \eqref{BD-density}, we show the spectral representation of the first hitting time densities at $0$ as follows:
\begin{Prop} \label{BD-hitting_density}
	For $i \geq 1$, it holds
	\begin{align}
		(f_i(t) :=)\ \partial_t\bP_i[T_0 \leq t] =  \pi_{i}\int_{0}^{\infty}\mathrm{e}^{-\theta t} \psi_{-\theta}(i)\rho(d\theta). \label{hitting-density_BD}
	\end{align}
\end{Prop}

\begin{proof}
	Note that $\bP_i[T_0 \leq t] = \bP_i[X_t = 0]$.
	From \eqref{F-eq} with $\lambda_0 = \mu_0 = 0$, \eqref{BD-density} and $\psi_{-\theta}(1) = C(1,1) = s(1) = 1/\mu_{1}$, we have for $i \geq 1$
	\begin{align}
		\partial_t P_t(i,0) &= \mu_1 P_t(i,1) \label{eq116} \\
		&= \pi_{i}\mu_1\int_{0}^{\infty}\mathrm{e}^{-\theta t}\psi_{-\theta}(i)\psi_{-\theta}(1)\rho(d\theta) \label{} \\
		&= \pi_{i}\int_{0}^{\infty}\mathrm{e}^{-\theta t}\psi_{-\theta}(i)\rho(d\theta). \label{BD-hitting-density}
	\end{align}
\end{proof}

\section{Proof of the reproduction formula} \label{section:proof_rep_formula}

We prove Theorem \ref{rep_transition}.

\begin{proof}[Proof of Theorem \ref{rep_transition}]
	Let us first consider the case where $\nu$ is a point mass.
	From Proposition \ref{eigen} and \ref{BD-hitting_density}, we can see for $i \geq 1$ that
	\begin{align}
		\psi_{\partial_t}(j)f_{i}(t) &= \sum_{k = 1}^{j}C(j,k){\partial_t}^{k-1}\int_{0}^{\infty}\mathrm{e}^{-\theta t}\psi_{-\theta}(i)\rho(d\theta) \label{} \\
		&= \sum_{k = 1}^{j}C(j,k)\int_{0}^{\infty}\mathrm{e}^{-\theta t}(-\theta)^{k-1}\psi_{-\theta}(i)\rho(d\theta) \label{} \\
		&= \int_{0}^{\infty}\mathrm{e}^{-\theta t}\psi_{-\theta}(i)\psi_{-\theta}(j)\rho(d\theta). \label{}
	\end{align}
	Thus \eqref{eq111} holds when $\nu$ is a point mass.
	
	Let us consider the general case.
	Suppose we may find some constants $\{\alpha_k\}_{k \geq 0}$ such that
	\begin{align}
		\max_{i \geq 1}|\partial^k_t f_i (t)| \leq \alpha_k \quad (t > 0). \label{eq112}
	\end{align}
	Then we have
	\begin{align}
		\sum_{i=1}^{\infty}\nu\{i\}|\partial^k_t f_i (t)| \leq \alpha_k, \label{}
	\end{align}
	and we can change the order of
	the differentiation w.r.t.\ $t$ and the integration by $\nu$:
	\begin{align}
		\sum_{i=1}^{\infty}\nu\{i\}\partial^k_t f_i (t) = \partial^k_t \sum_{i=1}^{\infty}\nu\{i\} f_{\nu} (t). \label{}
	\end{align}
	Then it follows that
	\begin{align}
		\psi_{\partial_t}(j)f_{\nu}(t) = \sum_{i = 1}^{\infty}\nu\{i\}\psi_{\partial_t}(j)f_{i}(t) = \bP_{\nu}[X_t = j]. \label{}
	\end{align}
	Here the last equality follows from \eqref{eq113} for a point mass.
	
	Let us find
	such a sequence $\{\alpha_k\}_{k \geq 0}$ as \eqref{eq112} is satisfied.
	We construct it recursively.
	From \eqref{eq116}, it holds
	\begin{align}
		f_i(t) = \mu_{1}P_t(i,1) \leq \mu_{1}, \label{}
	\end{align}	
	and we may take $\alpha_{0} := \mu_{1}$.
	Let $k \geq 0$ and assume we have constants $\alpha_{0},\alpha_{1}, \cdots, \alpha_{k}$ satisfying \eqref{eq112}. 
	Then from \eqref{eq113} in the case when $\nu$ is a point mass, we have 
	\begin{align}
		C(k+2,k+2){\partial_t}^{k+1}f_i(t) = P_t(i,k+2) -\sum_{l=1}^{k+1}C(k+2,l){\partial_t}^{l-1}f_i(t) \label{}
	\end{align}
	for every $i \geq 1$. 
	Thus it follows
	\begin{align}
				C(k+2,k+2)|{\partial_t}^{k+1}f_i(t)| \leq 1 + \sum_{l=1}^{k+1}C(k+2,l)\alpha_{l-1}. \label{}
	\end{align}
	From \eqref{eq110} it holds $C(k+2,k+2) > 0$ and hence it holds
	\begin{align}
		|{\partial_t}^{k+1}f_i(t)| \leq \frac{1 + \sum_{l=1}^{k+1}C(k+2,l)\alpha_{l-1}}{C(k+2,k+2)}, \label{}
	\end{align}
	and we may take 
	\begin{align}
		\alpha_{k+1} := \frac{1 + \sum_{l=1}^{k+1}C(k+2,l)\alpha_{l-1}}{C(k+2,k+2)}. \label{}
	\end{align}
	The proof is complete.
\end{proof}

\section{Looking at birth-and-death processes as generalized diffusions} \label{section:GeneralizedDiff}

Generalized diffusions are the processes which unify birth-and-death processes and one-dimensional diffusions.
Here we recall some studies of generalized diffusions and consider the application to birth-and-death processes.
A main reference for the subject is Kotani and Watanabe \cite{KotaniWatanabe}.

We say that $w: [0,\infty) \to [0,\infty]$ is a string when it is non-decreasing and right-continuous. Set $\ell(w) := \inf \{x \in \bR \mid w(x) = \infty\}$.
For a string $w$, we define the measure $dw$ on $[0,\ell(w))$ by $dw(a,b] = w(b) - w(a) \ (0 \leq a < b)$ and $dw\{0\} = w(0)$. 

Let $w$ be a string with and let $B$ be a standard Brownian motion and let $L(t,x)$ be the jointly continuous local time of $B$, that is, for every non-negative measurable function $f$ it holds
\begin{align}
	\int_{0}^{t}f(B_s)ds = 2\int_{\bR}f(x)L(t,x)dx. \label{}
\end{align}
Define
\begin{align}
	A(t) := \left\{
	\begin{aligned}
		&\int_{0}^{\ell(w)}L(t,x)dw(x) & (t < T_0 \wedge T_b), \\
		&\infty & (t \geq T_0 \wedge T_b).
	\end{aligned}
	\right.
	\label{time-change}
\end{align}
Then the process $X(t) := B(A^{-1}(t))$ is a strong Markov process on $I$ stopped at the boundaries. We call $X$ the $\frac{d}{dw}\frac{d}{dx}${\it -generalized diffusion} and $dw$ the {\it speed measure} of $X$.
Note that if the $dw$ has a full support in $(0,b)$, the process $X$ is a diffusion, and if $dw$ is supported on $\bN$, the process $X$ is a birth-and-death process.

It is known that the transition density of $X$ has the spectral representation. There exists a jointly continuous function $p(t,x,y)$ 
and a Radon measure $\rho_{w}$ on $(0,\infty)$ such that
\begin{align}
	\bP_x[X_t \in dy] = p(t,x,y)dw(y) \quad (t > 0,x,y \in (0,b)), \label{}
\end{align}
and
\begin{align}
	p(t,x,y) = \int_{0}^{\infty}\mathrm{e}^{-\lambda t}\psi_{-\lambda}(x)\psi_{-\lambda}(y)\rho_{w}(d\lambda) \quad  (t > 0,x,y \in (0,b)), \label{diffusion_transition}
\end{align}
(see e.g., McKean \cite{McKean:elementary}).
We call the measure $\rho_{w}$ the spectral measure of $\frac{d}{dw}\frac{d}{dx}$.

\subsection{Strings and spectral measures} \label{section:krein_theory}

For a string $w$, its dual string $w^{-1}$ defined by
\begin{align}
	w^{-1}(x) := \inf \{ y > 0 \mid w(y) > x \} \quad (x \geq 0) \label{}
\end{align}
is also a string.

We denote by $\rho_{w}$ (resp. $\sigma_{w}$) the spectral measure of $\frac{d}{dw}\frac{d}{dx}$ with Dirichlet (resp. Neumann) boundary condition at $0$, and if the boundary $\ell(w)$ is regular, we assume the Dirichlet boundary condition at $\ell(w)$. 
Then it is known that the following relation holds: 
\begin{align}
	\rho_{w}(d\lambda) = \lambda \sigma_{w^{-1}}(d\lambda)  \quad \text{on} \ [0,\infty) \label{spectral-correspondence}
\end{align}
(see Yano \cite[Theorem 2.2]{Yano:Excusionmeasure}).
Thus we can obtain the spectral measure from the one corresponding to the dual string.
This relation is useful because $\sigma_{w}$ can be sometimes obtained by the spectral theory, which we will explain below.

Let $w$ be a string.
For $\lambda \in \bC$, define $u = \varphi_\lambda$ as the unique solution of the following equation: 
\begin{align}
	u(x) = 1 + \lambda \int_{0}^{x}dy\int_{0-}^{y}u(z)dw(z) \quad (x \geq 0), \label{phi}
\end{align}
and define $u = \psi_{\lambda}$ as the unique solution of the following equation:
\begin{align}
	u(x) = x + \lambda \int_{0}^{x}dy\int_{0}^{y}u(z)dw(z) \quad (x \geq 0). \label{psi}
\end{align}
Define the spectral characteristic function $h$ by
\begin{align}
	h(\lambda) := \int_{0}^{\ell(w)}\frac{dx}{\varphi_\lambda(x)^2} \quad (\lambda  > 0). \label{}
\end{align}
Note that the following equality holds:
\begin{align}
	\int_{0}^{\ell(w)}\frac{dx}{\varphi_\lambda(x)^2} = \lim_{x \to \ell(w)}\frac{\psi_{\lambda}(x)}{\varphi_\lambda(x)}. \label{eq115}
\end{align}
Indeed, since the functions $\varphi_\lambda$ and $\psi_{\lambda}$ are $\lambda$-eigenfunctions for $\frac{d}{dw}\frac{d^+}{dx}$, it holds
\begin{align}
	d\left(\left(\frac{d^+}{dx}\psi_{\lambda}\right)\varphi_\lambda - \left(\frac{d^+}{dx}\varphi_\lambda\right)\psi_{\lambda}\right) = 0, \label{}
\end{align}
and it is clear that $\left.\left(\frac{d^+}{dx}\psi_{\lambda}\right)\varphi_\lambda - \left(\frac{d^+}{dx}\varphi_\lambda\right)\psi_{\lambda}\right|_{x=0} = 1$.
Thus $\left(\frac{d^+}{dx}\psi_{\lambda}\right)\varphi_\lambda - \left(\frac{d^+}{dx}\varphi_\lambda\right)\psi_{\lambda} = 1$.
Thus it follows
\begin{align}
	d\left(\frac{\psi_{\lambda}}{\varphi_\lambda}\right) = \frac{(\frac{d^+}{dx}\psi_{\lambda})\varphi_\lambda - (\frac{d^+}{dx}\varphi_\lambda)\psi_{\lambda}}{\varphi_\lambda^2}dx = \frac{dx}{\varphi_\lambda^2}, \label{}
\end{align}
and \eqref{eq115} holds.
By the spectral theory for generalized second-order differential operators, it is known that the function $h$ is represented by
\begin{align}
	h(\lambda) = c + \int_{0-}^{\infty}\frac{\sigma_{w}(d\xi)}{\lambda + \xi} 
	\quad (\lambda > 0) \label{char_func}
\end{align} 
for $c = \inf \Supp dw$. Here we note that $\sigma_{w}\{ 0 \} = 1 / w(\infty)$ (see \cite[p.239]{KotaniWatanabe}). 
Thus we can obtain $\sigma_{w}$ by the Stieltjes inversion formula:
\begin{align}
	\sigma_{w}(I) = -\frac{1}{\pi}\lim_{\eps \to + 0} \int_{I}\Im h(-\lambda + i \eps)d \lambda \quad \text{for }I \subset [0,\infty) \text{ with } \sigma_{w}(\partial I) = 0. \label{} 
\end{align} 

\subsection{Spectral measures of birth-and-death processes} \label{section:dual_process}

Let us consider a birth-and-death process $X$ whose $Q$-matrix is given by \eqref{q-matrix}. To apply the theory of generalized diffusions, we need a string which defines the generalized diffusion equivalent to $X$. 
For the scale function $s$ defined in \eqref{scale-function}, we denote its linear interpolation by the same symbol:
\begin{align}
	s(x) := (i+1-x)s(i) + (x-i)s(i+1) \quad \text{for} \quad i \geq 0 \quad \text{such that} \quad i \leq x < i+1. \label{}
\end{align}
Define a string $m$ on $(0,\infty)$ by
\begin{align}
	m(x) = \sum_{i=1}^{\infty}\pi_{i} 1 \{ i \leq x \} \quad (x > 0), \label{}
\end{align}
where $\pi_{i}$ is given in \eqref{speed-measure}.
Then 
\begin{align}
	w := m \circ s^{-1} \label{stringmBD}
\end{align}
defines a string and the generalized diffusion $\tilde{X}$ corresponding to $w$
has its state space $\{ s(i) \}_{i \geq 0}$.

We check that the generalized diffusion $s^{-1}(\tilde{X})$ is a realization of $X$.
Recalling that the generalized diffusions are obtained by the time change of a Brownian motion by $A^{-1}$ for $A$ in \eqref{time-change}, we can see
\begin{align}
	\bP_{s(i)}[\tilde{X}_{\tau} = s(i+1)] &= \frac{1/(\pi_{i-1}\lambda_{i-1})}{1/(\pi_{i}\lambda_{i}) + 1/(\pi_{i-1}\lambda_{i-1})} = \frac{\lambda_{i}}{\lambda_{i} + \mu_{i}} \quad (i \geq 1), \label{eq127} \\
	\bP_{s(i)}[\tilde{X}_{\tau} = s(i-1)] &= 1 - \bP_{s(i)}[\tilde{X}_{\tau} = s(i+1)] = \frac{\mu_{i}}{\lambda_{i} + \mu_{i}} \quad (i \geq 1), \label{eq128}
\end{align}
where $\tau := \inf \{t > 0 \mid \tilde{X}_t \neq \tilde{X}_0\}$ denotes the time of the first jump (see e.g., \cite[Theorem 23.7]{Kallenberg} and \cite[Proposition II.3.8]{RevusYor}).
Since $\tau$ is exponentially distributed, we can specify the distribution by the mean. Recall the following well-known formula (see e.g., \cite[Lemma 23.10]{Kallenberg}):
\begin{align}
	\bE_{x}\left[\int_{0}^{T_{a}\wedge T_{b}}f(\tilde{X}_t)dt\right] = \int_{\Supp dw \cap (a,b)}\frac{(x \wedge y - a)(b - x \vee y ) }{b-a}f(y)dw(y) \label{} \\
	 (a,b,x \in \Supp dw, a < x < b), \nonumber
\end{align}
where $f$ is non-negative measurable function.
Applying this formula for $x = s(i),\ a = s(i-1),\ b = s(i+1)$ and $f = 1$, we have
\begin{align}
	\bE_{s(i)}\tau &= \frac{(s(i) - s(i-1))(s(i+1)- s(i)) }{s(i+1) - s(i-1)}\pi_i = (\lambda_{i} + \mu_{i})^{-1} \quad (i \geq 1). \label{eq129}
\end{align}
Thus from \eqref{eq127}, \eqref{eq128} and \eqref{eq129}, we see that the process $\tilde{X}$ is equivalent to $X$.


From \eqref{spectral-correspondence} and \eqref{char_func}, we can obtain the spectral measure $\rho_{w}$ of $\tilde{X}$ (or $X$) through the eigenfunctions of $\frac{d}{dw^{-1}}\frac{d^{+}}{dx}$.
Since the eigenfunctions of $\frac{d}{dw^{-1}}\frac{d^{+}}{dx}$ can be computed from those of $\frac{d}{dw}\frac{d^{+}}{dx}$, we have a formula to compute $\rho_{w}$ from the eigenfunctions of $\frac{d}{dw}\frac{d^+}{dx}$. Here we go into the details.

It is not difficult to see that the dual string $w^{-1}$ of $w$ is equal to $s \circ m^{-1}$ and for $x \geq 0$ we have
\begin{align}
	w^{-1}(x) = s(i) \quad \text{for} \quad i \geq 1 \quad \text{such that} \quad m(i-1) \leq x < m(i). \label{}
\end{align}
Let $\varphi^{(-)}_\lambda$ and $\psi^{(-)}_{\lambda}$ be functions given in \eqref{phi} and $\eqref{psi}$ for $w = w^{-1}$, respectively.
Then as we have seen in the previous section, it holds
\begin{align}
	\lim_{x \to \ell(w^{-1})}\frac{\psi^{(-)}_{\lambda}(x)}{\varphi^{(-)}_{\lambda}(x)} &= \int_{0-}^{\infty}\frac{\sigma_{w^{-1}}(d\xi)}{\lambda + \xi} \label{} \\
	 &= \frac{1}{\lambda\ell(w)} + \int_{0}^{\infty}\frac{\rho_{w}(d\xi)}{\xi(\lambda + \xi)} \quad (\lambda > 0), \label{}
\end{align}
where we used the obvious fact $w^{-1}(\infty) = \ell(w)$.
Note that $\inf \Supp dw^{-1} = 0$ since $w^{-1}(0) = s(1) > 0$.
We consider the eigenfunctions of $\frac{d}{dx}\frac{d}{dw}$, which we denote by $u=\varphi^{d}_{\lambda}$, $v=\psi^{d}_{\lambda} \ (\lambda \in \bC)$, defined as the unique solutions of the following equations, respectively:
\begin{align}
	u(x) = 1 + \lambda \int_{0}^{x}dw(y)\int_{0}^{y}u(z)dz \quad (x \geq 0) \label{}
\end{align}
and
\begin{align}
	v(x) = w(x) + \lambda \int_{0}^{x}dw(y)\int_{0}^{y}v(z)dz.  \quad (x \geq 0). \label{}
\end{align}
It may not be difficult to check that
\begin{align}
	\varphi^{(-)}_{\lambda}(w(x)) = \varphi^{d}_{\lambda}(x) \quad \text{and} \quad \psi^{(-)}_{\lambda}(w(x)) = \psi^{d}_{\lambda}(x)  \quad (x \geq 0). \label{} 
\end{align}
We may also easily see that
\begin{align}
	\varphi^{d}_{\lambda}(x) = \psi^{+}_{\lambda}(x) \quad \text{and} \quad \psi^{d}_{\lambda}(x) = \frac{1}{\lambda}\varphi^{+}_{\lambda}(x) \quad (x \geq 0), \label{}
\end{align}
where $\varphi_\lambda$ and $\psi_{\lambda}$ are the solutions of \eqref{phi} and \eqref{psi} for $w$ in \eqref{stringmBD}, respectively.

Eventually, it follows for $\lambda> 0$
\begin{align}
	\lim_{x \to \ell(w)}\frac{\varphi^{+}_{\lambda}(x)}{\psi^{+}_{\lambda}(x)} = \lim_{x \to \ell(w)}\frac{\lambda\psi^{d}_\lambda(x)}{\varphi^{d}_{\lambda}(x)} = \lim_{x \to \ell(w)}\frac{\lambda\psi^{(-)}_{\lambda}(w(x))}{\varphi^{(-)}_{\lambda}(w(x))}
	=\frac{1}{\ell(w)} + \lambda \int_{0}^{\infty}\frac{\rho_{w}(d\xi)}{\xi(\lambda+\xi)}. \label{exitKrein}
\end{align}
Thus, we obtain the Stieltjes transform of $\xi^{-1}\rho_{w}(d\xi)$ through the eigenfunctions $\varphi_{\lambda}$ and $\psi_{\lambda}$. By the Stieltjes inversion formula, we can compute $\rho_{w}$ in some cases. See Section \ref{section:exBD} for such examples.

\subsection{Doob's $h$-transform}

Here we recall some basic properties of $h$-transform for generalized diffusions (see e.g., Takemura and Tomisaki \cite{TakemuraTomisaki:htransform} for details).
Since we are interested in the first hitting time of $0$, we restrict our attention to the case of generalized diffusions corresponding to a string $w$ with $\ell(w) = \infty$.

Let $w$ be a string and let $\gamma \geq 0$.
Suppose $k_{\gamma}$ be a positive $\gamma$-eigenfunction for $\frac{d}{dw}\frac{d^+}{dx}$, that is, $k_{\gamma}(x) > 0 \ (x > 0)$ and
$u = k_{\gamma}$ is the unique solution of the following equation:
\begin{align}
	u(x) = a + b(x - c) + \gamma \int_{c}^{x}dy\int_{c-}^{y}u(z)dw(z) \quad (x > 0) \label{integral-eq}
\end{align}
for some $a, b \in \bR$ and $c > 0$.
Here we note that for $\gamma \geq 0$ the function 
\begin{align}
	g_\gamma(x) := \psi_{\gamma}(x)\int_{x}^{\infty} \frac{dy}{\psi_{\gamma}(y)^2} \quad (x \geq 0) \label{}
\end{align}
is a unique non-increasing solution of \eqref{integral-eq} satisfying
\begin{align}
	u(0) = 1, \quad \lim_{x \to \infty}\frac{d^+}{dx}u(x) = 0 \label{}
\end{align}
(see e.g., It\^o \cite{Ito_essentials} for details).

The $h$-transform $\cL^{[k_{\gamma}]}$ of $ \cL = \frac{d}{dw}\frac{d}{dx}$ by $k_{\gamma}$ is defined by
\begin{align}
	\cL^{[k_{\gamma}]} := \frac{1}{k_{\gamma}}\left(\frac{d}{dw}\frac{d}{dx} - \gamma \right)k_{\gamma}, \label{}
\end{align}
and we can easily check that
\begin{align}
	\cL^{[k_{\gamma}]} = \frac{d}{dw^{[k_{\gamma}]}}\frac{d}{ds^{[k_{\gamma}]}} \label{}
\end{align}
for
\begin{align}
	dw^{[k_{\gamma}]} = k_{\gamma}^2 dw, \quad ds^{[k_{\gamma}]} = k_{\gamma}^{-2} dx. \label{eq117}
\end{align}
Note that when $k_{\gamma}(0) > 0$, the boundary classification for the boundary $0$ is not changed by the transform. As for the boundary $\infty$, it is entrance or natural depending on
\begin{align}
	\int_{1}^{\infty}k_{\gamma}(x)^{-2}dx\int_{x}^{\infty}k_{\gamma}(x)^2dw(x) \label{}
\end{align}
is finite or not.

Let us consider the case of $k_{\gamma}(0) = 1$. When we denote the functions $\psi_{\lambda}$ and $g_\lambda$ for $\cL^{[k_{\gamma}]} \ ( \gamma \geq 0)$ by $\psi_{\lambda}^{[k_{\gamma}]}$ and $g_\lambda^{[k_{\gamma}]}$,
it holds
\begin{align}
	\psi_{\lambda}^{[k_{\gamma}]} = \frac{\psi_{\lambda + \gamma}}{k_{\gamma}}, \quad g_\lambda^{[k_{\gamma}]} = \frac{g_{\lambda + \gamma}}{k_{\gamma}} \quad (\lambda \in \bC). \label{eigen_shift}
\end{align}
Moreover, for the spectral measure $\rho_{w^{[k_{\gamma}]}}$ of $\cL^{[k_{\gamma}]}$ is given by
\begin{align}
	\rho_{w^{[k_{\gamma}]}}(d\lambda) = \rho_{w}(d(\lambda - \gamma)) \quad \text{on} \ (\gamma, \infty). \label{spectral_shift}
\end{align}
Thus the $h$-transformed transition probability subject to the absorbing boundary at $0$ and the first hitting time densities are
\begin{align}
	P^{[k_\gamma]}_t(x,dy) = \mathrm{e^{-\gamma t}}\frac{k_{\gamma}(y)}{k_\gamma(x)} P_t(x,dy)
	\quad \text{and} \quad
	f^{[k_{\gamma}]}_x(t) = \frac{\mathrm{e}^{-\gamma t}}{k_{\gamma}(x)}f_x(t). \label{h-transformed_density}
\end{align}

\subsection{$h$-transform for birth-and-death processes}

Let $w$ be a string such that $dw$ is supported on a discrete countable set $\{a_i\}_{i \geq 1} \quad (0 <  a_1 < a_2 < \cdots)$; $dw(x) = \sum_{i = 1}^{\infty}\pi_{i}\delta_{a_i} \quad (\pi_{i} > 0)$. 
As we have seen in Section \ref{section:dual_process}, the generalized diffusion corresponding to $w$ is equivalent to a birth-and-death process.

For $\gamma \geq 0$, let $k_\gamma$ be a $\gamma$-eigenfunction for $\frac{d}{dw}\frac{d}{dx}$:
\begin{align}
	k_\gamma(x)  =a + bx + \gamma \int_{0}^{x}dy\int_{0}^{y}k_\gamma(z)dw(z) \quad (x > 0) \label{eq126}
\end{align}
for some constants $a, b\in \bR$.
Note that since the boundary $0$ for birth-and-death processes is always regular, all the $\gamma$-eigenfunction may be represented of the form \eqref{eq126}.

Set $a_0 := 0$. The function $k_\gamma$ is linear on each interval $[a_i,a_{i+1}] \ (i = 0,1,2,\cdots)$. Indeed, for $i = 0,1,2,\cdots,$ and $\delta \in [0,(a_{i+1} - a_i))$, since $dw$ is supported on $\{ a_i \}_{i \in \bN}$, it holds
\begin{align}
	k_{\gamma}(a_i + \delta) - k_{\gamma}(a_i)
	&= b \delta + \gamma \int_{a_i}^{a_i+\delta}dy\int_{0}^{y}k_\gamma(z)dw(z) \label{} \\
	&=b\delta 
	+ \gamma \int_{a_i}^{a_i+\delta}dy\int_{0}^{a_i}k_\gamma(z)dw(z) \label{} \\
	&=b\delta 
	+ \gamma \delta \int_{0}^{a_i}k_\gamma (z)dw(z) \label{} \\
	&=\left( b + \gamma \int_{0}^{a_i}k_\gamma(z)dw(z) \right)\delta. \label{}
\end{align} 
Suppose $k_{\gamma}(x) > 0$ for all $x \geq 0$.
For the scale function $s^{[k_\gamma]}$ defined by \eqref{eq117}, we have
\begin{align}
	s^{[k_{\gamma}]}(a_{i+1}) - s^{[k_{\gamma}]}(a_i) &= \int_{a_i}^{a_{i+1}}k_{\gamma}(x)^{-2}dx \label{} \\
	&= \int_{a_i}^{a_{i+1}}\frac{(a_{i+1} - a_{i})^2}{(x(k_{\gamma}(a_{i+1}) - k_{\gamma}(a_i)) + a_{i+1}k_{\gamma}(a_i) - a_ik_{\gamma}(a_{i+1}))^2}dx \label{} \\
	&= \frac{a_{i+1} - a_{i}}{ k_{\gamma}(a_i)k_{\gamma}(a_{i+1})}. \label{BD-h-transformed_s} 	
\end{align}

For the birth-and-death process corresponding to the $h$-transformed generalized diffusion, we show the birth and death rates $\{\lambda_i^{[k_\gamma]}\}$ and $\{ \mu^{[k_\gamma]}_i \}$.
By the same argument in Section \ref{section:dual_process}, we see that
\begin{align}
	\frac{\lambda_i^{[k_\gamma]}}{\mu_i^{[k_\gamma]}} = \frac{s^{[k_\gamma]}(a_{i}) - s^{[k_\gamma]}(a_{i-1})}{s^{[k_\gamma]}(a_{i+1}) - s^{[k_\gamma]}(a_{i})} = \frac{a_{i} - a_{i-1}}{a_{i+1} - a_{i}}\cdot\frac{k_{\gamma}(a_{i+1})}{k_\gamma(a_{i-1})} \label{}
\end{align}
and
\begin{align}
	(\lambda_i^{[k_\gamma]} + \mu_i^{[k_\gamma]})^{-1} &= \frac{(s^{[k_\gamma]}(a_{i}) - s^{[k_\gamma]}(a_{i-1}))(s^{[k_\gamma]}(a_{i+1}) - s^{[k_\gamma]}(a_{i}))}{s^{[k_\gamma]}(a_{i+1}) - s^{[k_\gamma]}(a_{i-1})} k_\gamma(a_{i})^2\pi_{i}. \label{}
\end{align}
Solving this, we obtain
\begin{align}
	\lambda_i^{[k_\gamma]} = \frac{1}{\pi_{i}(a_{i+1} - a_{i})}\cdot\frac{k_{\gamma}(a_{i+1})}{k_{\gamma}(a_{i})} \quad \text{and} \quad
	\mu_i^{[k_\gamma]} = \frac{1}{\pi_{i}(a_{i} - a_{i-1})}\cdot\frac{k_{\gamma}(a_{i-1})}{k_{\gamma}(a_{i})}. \label{}
\end{align}
We will apply this formula in Section \ref{section:exBD}.

We can see how the matrix $C$ is changed under an $h$-transform.
Using \eqref{eigen_shift}, we can compute the $\theta$-eigenfunction $\psi_{\theta}^{[k_\gamma]}$ as  
\begin{align}
	\psi^{[k_\gamma]}_{\theta}(i) &= \frac{\psi_{\theta + \gamma}(i)}{k_{\gamma}(i)}  \label{} \\
	&= \frac{1}{k_{\gamma}(i)}\sum_{l=1}^{i}C(i,l)(\theta+\gamma)^{l-1} \label{} \\
	&= \frac{1}{k_{\gamma}(i)}\sum_{l=1}^{i}C(i,l)\sum_{j = 1}^{l} \binom{l-1}{j-1}  \theta^{j-1}\gamma^{l-j} \label{} \\
	&= \frac{1}{k_{\gamma}(i)}\sum_{j = 1}^{i}\theta^{j-1} \sum_{l=j}^{i} \binom{l-1}{j-1} C(i,l) \gamma^{l-j}. \label{}
\end{align}
Thus we now determine the matrix $C$ corresponding to the $h$-transformed birth-and-death process, which we denote by $C^{[k_\gamma]}$:
\begin{align}
	C^{[k_\gamma]}(i,j) = \frac{1}{k_{\gamma}(i)}\sum_{l=j}^{i} \binom{l-1}{j-1} C(i,l) \gamma^{l-j}. \label{h-transformed_C-matrix}
\end{align}

\section{Examples} \label{section:exBD}

Here we give examples of birth-and-death processes and compute the corresponding matrices $C$.

\subsection{Symmetric random walk} \label{ex:sym-RW}
Let us consider the case $\lambda_{i} = \mu_i = \kappa > 0$ for every $i \geq 1$.
In this case,
\begin{align}
	\pi_{i} = 1 \quad (i \geq 1), \quad s(i) = i/\kappa \quad (i \geq 0) \label{}
\end{align}
and
\begin{align}
	Qf(i) = \kappa (f(i+1) - 2 f(i) + f(i-1) ) \quad (i \geq 1). \label{}
\end{align}
Solving the recurrence relation $Qf = \theta f $ for $\theta \in \bR$, we have the following linearly independent solutions $k_\pm$:
\begin{align}
	k_\pm(i) := \alpha^{\kappa}_{\pm}(\theta)^{i} \quad \text{for} \quad \alpha^{\kappa}_{\pm}(\theta) = \left( 1 + \frac{\theta}{2\kappa} \right) \pm \sqrt{\left( 1 + \frac{\theta}{2\kappa} \right)^2 -1} \quad (i \geq 0), \label{eq118}
\end{align}  
where we note that $\alpha^{\kappa}_\pm(\theta)$ are the solutions of the quadratic equation: $\kappa(\alpha^2 - 2\alpha + 1) = \theta \alpha$.
We note that $\alpha^{\kappa}_+(\theta)\alpha^{\kappa}_{-}(\theta) = 1$.
Thus, it follows
\begin{align}
	\varphi_\theta(i) = \left(\frac{1}{2} - \frac{\theta}{2\sqrt{\theta^2 + 4\kappa \theta}}\right) \alpha^{\kappa}_+(\theta)^i + \left(\frac{1}{2} + \frac{\theta}{2\sqrt{\theta^2 + 4\kappa \theta}}\right) \alpha^{\kappa}_-(\theta)^i \label{}  
\end{align}
and
\begin{align}
	\psi_{\theta}(i) = \frac{1}{\sqrt{\theta^2 + 4\kappa \theta}}(\alpha^{\kappa}_+(\theta)^i - \alpha^{\kappa}_-(\theta)^i). \label{}
\end{align}
Since it holds
\begin{align}
	\lim_{i \to \infty}\frac{\varphi^{+}_{\theta}(i)}{\psi^{+}_\theta(i)} 
	= \lim_{i \to \infty} \frac{\varphi_{\theta}(i+1) - \varphi_{\theta}(i)}{\psi_{\theta}(i+1) - \psi_{\theta}(i)}
	= \frac{2\kappa \theta}{\theta + \sqrt{\theta^2 + 4\kappa \theta}} \quad (\theta > 0), \label{}
\end{align}
it follows from \eqref{exitKrein} that
\begin{align}
	\frac{2\kappa}{\theta + \sqrt{\theta^2 + 4\kappa \theta}} = \int_{0}^{\infty}\frac{\rho_{w}(d\xi)}{\xi(\theta + \xi)} \quad (\theta > 0), \label{}
\end{align}s


By the Stieltjes inversion, we have
\begin{align}
	\rho_{w}(d\theta) = \frac{\sqrt{\theta(4\kappa - \theta)}}{2\pi}1\{ 0 < \theta < 4\kappa\}d\theta. \label{}
\end{align}
For $\theta \in (0,4\kappa)$ we can easily see
\begin{align}
	\psi_{-\theta}(i) = \frac{\sin (\beta(\theta) i)}{\kappa \sin (\beta(\theta))} = \frac{1}{\kappa}U_{i-1}(\cos (\beta(\theta))), \label{}
\end{align}
where $\beta(\theta) = \arctan \left( \frac{\sqrt{4\kappa\theta - \theta^2}}{2\kappa - \theta}\right)$ and $U_{i-1}$ is the $(i-1)$-th Chebyshev polynomial of the second kind.
Note that the Chebyshev polynomials are characterized by
\begin{align}
	U_{k}(\cos t) = \frac{\sin ((k+1)t)}{\sin t} \quad (k \geq 0, \ t \in \bR) \label{}
\end{align}
(see e.g., \cite[p.218]{Specialfunction}).
Note that $\cos (\beta(\theta)) = 1 - \frac{\theta}{2\kappa}$.
When we write
\begin{align}
	U_{k}(\theta) = \sum_{l = 0}^{k}u(k,l)\theta^l, \label{}
\end{align}
it holds
\begin{align}
	\psi_{\theta}(i) &= \frac{1}{\kappa}U_{i-1}\left(1 + \frac{\theta}{2\kappa}\right) \label{} \\
	&= \frac{1}{\kappa}\sum_{l=1}^{i}u(i-1,l-1)(1 + \theta /2\kappa)^{l-1} \label{} \\
	&= \frac{1}{\kappa}\sum_{l=1}^{i}u(i-1,l-1)\sum_{j=1}^{l}\binom{l-1}{j-1} \left(\frac{\theta}{2\kappa}\right)^{j-1} \label{} \\
	&= \frac{1}{\kappa}\sum_{j=1}^{i} \left(\frac{\theta}{2\kappa}\right)^{j-1} \sum_{l=j}^{i}u(i-1,l-1)\binom{l-1}{j-1}. \label{}
\end{align}
Thus it follows
\begin{align}
	C(i,j) =  \frac{1}{2^{j-1}\kappa^{j}}
	\sum_{l=j}^{i} \binom{l-1}{j-1} u(i-1,l-1). \label{}
\end{align}
Note that from \cite[p.219]{Specialfunction}, it holds
\begin{align}
	u(i,j) = 
	\begin{cases}
		(-1)^{n-k} \binom{n+k}{n-k}2^{2k} & (i = 2n,\ j = 2k, \ 0 \leq  k \leq n), \\
		(-1)^{n-k} \binom{n+k+1}{n-k}2^{2k+1} & (i = 2n+1,\  j = 2k+1, \ 0 \leq  k \leq n), \\
		0 & \text{otherwise.}
	\end{cases}
	\label{}
\end{align}

\subsection{Asymmetric random walk} \label{Ex:asymRW}

The asymmetric random walk with birth rate $\lambda$ and death rate $\mu$ can be obtained from the symmetric one by $h$-transform.
We keep the notation of Section \ref{ex:sym-RW}.
Recall that for $\theta > 0$, $k_+(i)$ and $k_-(i)$ are positive increasing and decreasing $\theta$-eigenfunctions given in \eqref{eq118}, respectively.

Fix $\gamma \geq 0$. From \eqref{BD-h-transformed_s}, the speed measure $\pi^{\pm}$ and scale function $s^{\pm}$ of the symmetric random walk of the previous section transformed by $k_{\pm}(i) := \alpha^{\kappa}_{\pm}(\gamma)^{i} $ are
\begin{align}
	\pi^{\pm}_i = \alpha^{\kappa}_{\pm}(\gamma)^{2i}, \quad s^{\pm}(i) - s^{\pm}(i-1) = \kappa^{-1}\alpha^{\kappa}_{\pm}(\gamma)^{-2i+1} \quad (i \geq 1). \label{h-ms}
\end{align}
Thus its birth and death rates are given by
\begin{align}
	\lambda^{\pm}_{i} = \kappa \alpha^{\kappa}_{\pm}(\gamma), \quad \mu^{\pm}_i = \kappa \alpha^{\kappa}_{\pm}(\gamma)^{-1} = \kappa\alpha^{\kappa}_{\mp}(\gamma) \quad (i \geq 0). \label{h-lm}
\end{align}

The transition probability and the first hitting time densities can be easily obtained by \eqref{h-transformed_density}.
We can also compute the matrix $C$ corresponding to these processes, which we denote by $C^{\pm}_{\gamma}$.
By \eqref{h-transformed_C-matrix}, we have
\begin{align}
	C^{\pm}_{\gamma}(i,j) &= \frac{1}{\alpha^{\kappa}_{\pm}(\gamma)^i}\sum_{l=j}^{i}\binom{l-1}{j-1}C(i,l)\gamma^{l-j}. \label{} \\
	&= \frac{1}{\alpha^{\kappa}_{\pm}(\gamma)^i}\sum_{l=j}^{i}\binom{l-1}{j-1}\frac{\gamma^{l-j}}{2^{l-1}\kappa^{l}}\sum_{m = l}^{i}\binom{m-1}{l-1}u(i-1,m-1) \label{} \\
	&= \frac{1}{\alpha^{\kappa}_{\pm}(\gamma)^i} 
	\sum_{m = j}^{i} u(i-1,m-1)
	\sum_{l=j}^{m} \binom{l-1}{j-1} \binom{m-1}{l-1} \frac{\gamma^{l-j}}{2^{l-1}\kappa^{l}}. \label{} 
\end{align}

We may represent every asymmetric random walk as an $h$-transform of a symmetric one.
Let $\lambda, \mu > 0$. Set $\kappa = \sqrt{\lambda \mu}$.
When $\lambda > \mu$, we can take $\gamma > 0$ so that
\begin{align}
	\alpha^{\kappa}_{+}(\gamma)^2 = \frac{\alpha^{\kappa}_+(\gamma)}{\alpha^{\kappa}_-(\gamma)} = \frac{\lambda}{\mu} \label{}
\end{align}
since $\alpha^{\kappa}_+ :[0,\infty) \to [1,\infty)$ is increasing and homeomorphic.
Then it holds
\begin{align}
	\lambda^+_i = \kappa \alpha^{\kappa}_+(\gamma) = \lambda
	\quad \text{and} \quad
	\mu^+_i = \kappa \alpha^{\kappa}_+(\gamma)^{-1} = \mu. \label{}
\end{align}

When $\lambda < \mu$, by the same way as above, we can take $\kappa = \sqrt{\lambda \mu}$ and $\gamma > 0$ such that $\alpha^{\kappa}_-(\gamma) = \lambda / \mu$. Then it holds
\begin{align}
	\lambda^-_i = \kappa \alpha^{\kappa}_{-}(\gamma) = \lambda \quad \text{and} \quad \mu^-_i = \kappa \alpha^{\kappa}_{-}(\gamma)^{-1} = \mu. \label{}
\end{align}

\appendix

\section{Reproduction of initial distributions for one-dimensional diffusions} \label{section:1-dimDiff}

\subsection{Initial distributions with square-integrable densities}

Let $X$ be a $\frac{d}{dm}\frac{d}{ds}$-diffusion on $[0,b)$ stopped at $0$.
We assume the boundary $0$ is accessible, that is, regular or exit.
Define the function $u = \psi_{\lambda}$ as the unique solution of the following differential equation:
\begin{align}
	\frac{d}{dm}\frac{d}{ds}u= \lambda u, \quad  u(0) = 0, \quad \frac{d}{ds}u(0) = 1 \quad (0 < x < b, \lambda \in \bR). \label{}
\end{align}
From the spectral theory (see e.g., Coddington and Levinson \cite{Coddington}), it is known that there exists a measure $\sigma$ on $[0,\infty)$ and for every function $f:(0,b) \to \bR$ such that
\begin{align}
	\int_{0}^{b}f(x)^2dm(x) < \infty \quad \text{and} \quad f(x) = 0 \quad (x \geq R) \quad \text{for some } R \in (0,b), \label{}
\end{align}
it holds
\begin{align}
	\int_{0}^{b}f(x)^2dm(x) = \int_{0}^{\infty}\hat{f}(\lambda)^2\sigma(d\lambda), \label{eq104}
\end{align}
where $\hat{f}$ is the generalized Fourier transform:
\begin{align}
	\hat{f}(\lambda) := \int_{0}^{b}f(x)\psi_{-\lambda}(x)dm(x).
\end{align}
We call $\sigma$ the spectral measure of the $\frac{d}{dm}\frac{d}{ds}$.
The transform has its inverse. For a function $g: (0,\infty) \to \bR$ with
\begin{align}
	\int_{0}^{\infty}g(\lambda)^2\sigma(d\lambda) < \infty \quad \text{and} \quad g(\lambda) = 0 \quad (\lambda \geq R) \quad \text{for some } R > 0, \label{}
\end{align}
we define
\begin{align}
	\check{g}(x) := \int_{0}^{\infty}g(\lambda)\psi_{-\lambda}(x)\sigma(d\lambda). \label{}
\end{align}
Then it holds
\begin{align}
	\int_{0}^{b}\check{g}(x)^2dm(x) = \int_{0}^{\infty}g(\lambda)^2\sigma(d\lambda). \label{eq106}
\end{align}
From \eqref{eq104}, these transforms are naturally extended to the unitary transforms between $L^2(dm)$ and $L^2(\sigma)$. We denote the extended ones by the same symbol $\hat{\ }$ and $\check{\ }$.

In this section, we assume the existence of the Laplace transform of the spectral measure:
\begin{align}
	\mathrm{(S)} \quad \int_{0}^{\infty}\mathrm{e}^{-\lambda t}\rho(d\lambda) < \infty \quad (t > 0). \label{}
\end{align}
A sufficient condition for $\mathrm{(S)}$ is the following:
\begin{Prop}[{\cite{report}}]
	Suppose $|s(0)| < \infty$ and
	\begin{align}
		m(x,c] \leq C (s(x) - s(0))^{-\delta} \quad (0 < x < c) \label{} 
	\end{align}
	for some $C > 0$, $0 < c < b$ and $0 < \delta < 1$. Then the condition $\mathrm{(S)}$ holds.
\end{Prop}

Under the condition $\mathrm{(S)}$, we may have the spectral representation of the first hitting time densities.
\begin{Prop}[{\cite[Proposition 2.1]{Yano:Excusionmeasure}}] \label{hitting_density}
	Assume $\mathrm{(S)}$ holds.
	Then for any $0 < x < b$ the distribution of $T_0$ under $\bP_{x}$ has a density $f_x(t)$ on $(0,\infty)$ w.r.t.\ the Lebesgue measure, that is, the following hold:
	\begin{align}
		\bP_x[T_0 \in dt] = f_x(t)dt  \quad (0 < x < b, \ t > 0). \label{eq19}
	\end{align}
	The first hitting time densities have a spectral representation:
	\begin{align}
		f_x(t) = \int_{0}^{\infty}\mathrm{e}^{-\lambda t}\psi_{-\lambda}(x)\rho(d\lambda) \quad (0 < x < b,\ t > 0). \label{eq20}
	\end{align}
	and have another representation:
	\begin{align}
		f_{x}(t) = \left.\frac{d}{ds(y)}p(t,x,y)\right|_{y=0} \quad (0 < x < b, \ t > 0). \label{}
	\end{align}
\end{Prop}

Let
\begin{align}
	\cP_2 := \{ \nu \in \cP(0,b)  \mid \nu = \eta dm \quad \text{for some } \eta \in L^2(dm)   \}.
\end{align}
Applying the spectral representation, we can reproduce the initial distribution on $\cP_2$.
Before we state the theorem, we recall the relation of the Laplace transform of the signed measures and the completely monotone functions.
Let $\cM$ be a set of signed measures on $(0,\infty)$ with the Laplace transforms, that is, for every $\mu \in \cM$ it holds 
\begin{align}
	\cL\mu(t) := \int_{0}^{\infty}\mathrm{e}^{-\lambda t}|\mu|(d\lambda) < \infty \quad (t > 0), \label{}
\end{align} 
where $|\mu|$ denotes the total variation measure of $\mu$.
We say that a function $f:(0,\infty) \to \bR$ is a completely monotone when 
$(-1)^kf^{(k)}(t) \geq 0 \ (k \geq 0, t > 0)$.
We denote by $\mathrm{CM}$ the set of the difference of two completely monotone functions:
\begin{align}
	\mathrm{CM} := \{ f - g \mid f,g: \ \text{completely monotone} \}. \label{}
\end{align}
By the classical Bernstein theorem (see e.g., \cite[Theorem 1.4]{bernstein_functions}), the Laplace transform $\cL$ is a bijection from $\cM$ to $\mathrm{CM}$. We denote by $\cL^{-1}: \mathrm{CM} \to \cM$ the inverse transform of $\cL$.

\begin{Thm}
	Assume the condition $\mathrm{(S)}$ holds.
	Let $\nu \in \cP_2$. Then $f_{\nu} \in \mathrm{CM}$ and it holds
	\begin{align}
		\frac{d\nu}{dm}(x) = \int_{0}^{\infty}\psi_{-\lambda}(x)(\cL^{-1}f_{\nu})(d\lambda)
		\quad m\text{-a.e.} \label{eq132}
	\end{align}
	Therefore we can reproduce the initial distribution.
\end{Thm}

\begin{proof}
	Let $\nu \in \cP_2$ with $\nu = \eta dm$ for $\eta \in L^2(dm)$.
	Set $\mathbf{e}_t(\lambda) := \mathrm{e}^{-\lambda t} \ (t > 0, \lambda > 0)$.
	Note that by $\mathrm{(S)}$, the function $\mathbf{e}_t(\lambda) \in L^2(\rho)$. 
	Then from Proposition \ref{hitting_density} and $\mathrm{(S)}$, it holds $f_x(t) = \check{\mathbf{e}_t}(x)$.
	Since the generalized Fourier transform is unitary, the first hitting time density $f_\nu$ under $\nu$ satisfies
	\begin{align}
		f_\nu(t) :&= \int_{0}^{b}f_x(t)\nu(dx)  \label{} \\
		&= \int_{0}^{b}\check{\mathbf{e}_t}(x)\eta(x)dm(x) \label{} \\
		&= \int_{0}^{\infty}\mathrm{e}^{-\lambda t}\hat{\eta}(\lambda)\rho(d\lambda). \label{}
	\end{align}
	Thus, it holds
	\begin{align}
		(\cL^{-1}f_{\nu})(d\lambda) = \hat{\eta}(\lambda)\sigma(d\lambda). \label{}
	\end{align}
	By taking the inverse of the generalized Fourier transform, we obtain \eqref{eq132}.
\end{proof}

\subsection{Examples}
Here we see examples of one-dimensional diffusions whose transition density, first hitting time densities and generalized
Fourier transform can be computed explicitly.

\subsubsection{Brownian motion}

Let us consider a standard Brownian motion $B$ on $[0,\infty)$ stopped at $0$.
Its corresponding string $\tilde{m}$ and the scale function $s$ can be given by
\begin{align}
	\tilde{m}(x) = 2x, \quad s(x) = x \quad (x > 0). \label{}
\end{align}
Then $m(x) := \tilde{m} \circ s^{-1}(x) = 2x$ and its dual string is $m^{-1}(x) := x/2$.
The eigenfunctions $\varphi^{\ast}_{\lambda}$ and $\psi^{\ast}_{\lambda} \ (\lambda \in \bC)$ for $\frac{d}{dm^{-1}}\frac{d}{dx} = 2\frac{d^2}{dx^2}$ is
\begin{align}
	\varphi^{\ast}_{\lambda}(x) &= \left\{
	\begin{aligned}
		&\cosh (\sqrt{\lambda / 2}x) & (\lambda > 0), \\
		&1  &  ( \lambda = 0), \\
		&\cos (\sqrt{-\lambda / 2} x) &  ( \lambda < 0)
	\end{aligned}
	\right.
	\label{}
\end{align}
and
\begin{align}
	\psi^{\ast}_{\lambda}(x) &= \left\{
	\begin{aligned}
		&\sqrt{2 / \lambda} \sinh (\sqrt{\lambda / 2} x)  &  ( \lambda > 0), \\
		&x  &  ( \lambda = 0), \\
		&\sqrt{-2 / \lambda} \sin (\sqrt{-\lambda / 2} x) &  ( \lambda < 0).
	\end{aligned}
	\right.
	\label{}
\end{align}
Thus the spectral characteristic function $h$ for $m^{-1}$ is
\begin{align}
	h(\lambda) = \sqrt{\frac{2}{\lambda}} \quad (\lambda > 0). \label{}
\end{align}
From Stieltjes inversion, the spectral measure $\sigma$ of $\frac{d}{dm^{-1}}\frac{d}{dx}$ is given by
\begin{align}
	\sigma(d\lambda) = \frac{1}{\pi} \sqrt{\frac{2}{\lambda}}d\lambda \quad (\lambda > 0), \label{}
\end{align}
and thus the spectral measure $\rho$ for $\frac{d}{dm}\frac{d}{dx}$ is 
\begin{align}
	\rho(d\lambda) = \frac{1}{\pi} \sqrt{2 \lambda}d\lambda \quad (\lambda > 0). \label{eq124}
\end{align}
Since the eigenfunction $\psi_{\lambda} \ (\lambda > 0)$ for $\frac{d}{dm}\frac{d}{dx}$ is
\begin{align}
	\psi_{-\lambda}(x) 
	= \frac{1}{\sqrt{2\lambda}} \sin (\sqrt{2\lambda} x), \label{eq123}
\end{align}
we have the spectral representation of the transition density w.r.t. $dm$:
\begin{align}
	p(t,x,y) &= \frac{1}{\pi \sqrt{2}}
	\int_{0}^{\infty}\mathrm{e}^{-\lambda t} \sin (\sqrt{2\lambda} x) \sin (\sqrt{2\lambda} y)\frac{d\lambda}{\sqrt{\lambda}} \label{} \\
	&= \frac{1}{\pi}
	\int_{0}^{\infty}\mathrm{e}^{-\lambda^2 t/2} \sin (\lambda x) \sin (\lambda y) d\lambda 
	\quad (t > 0, x ,y > 0). \label{}
\end{align}
From Proposition \ref{hitting_density}, the first hitting time densities are 
\begin{align}
	f_x(t) = \left.\frac{d^+}{ds(y)}p(t,x,y) \right|_{y=0} 
	= \frac{1}{\pi}
	\int_{0}^{\infty}\mathrm{e}^{-\lambda t} \sin (\sqrt{2\lambda} x) d\lambda. \label{} 
\end{align}
The generalized Fourier transform is given By
\begin{align}
	\hat{f}(\lambda) &= \int_{0}^{\infty}f(x)\psi_{-\lambda}(x)dm(x) \label{} \\
	&= \sqrt{\frac{2}{\lambda}}\int_{0}^{\infty}f(x)\sin (\sqrt{2\lambda}x)dx \quad (f \in L^2(dm), \ \lambda > 0). \label{}
\end{align}
Its inverse transform is
\begin{align}
	\check{g}(x) &= \int_{0}^{\infty}g(\lambda)\psi_{-\lambda}(x)\rho(d\lambda) \label{} \\
	&= \frac{1}{\pi}\int_{0}^{\infty}g(\lambda) \sin (\sqrt{2\lambda}x)d\lambda \quad (g \in L^2(\rho), \ x > 0). \label{}
\end{align}
\subsubsection{Brownian motion with a constant drift}

We consider a Brownian motion $B$ with a constant drift.
Similarly with the case of asymmetric random walks in Section \ref{Ex:asymRW}, the process $X$ can be obtained as an $h$-transform of a standard Brownian motion.

Fix $\gamma \geq 0$.
Then we may easily see that the functions
\begin{align}
	k_{\pm}(x) := \mathrm{e}^{\pm\sqrt{2\gamma} x} \quad (x > 0). \label{}
\end{align}
are the positive increasing and decreasing $\gamma$-eigenfunctions for $\frac{1}{2}\frac{d^2}{dx^2}$ and $k_\pm(0) = 1$, respectively.
The $h$-transformed speed measure $dm_\pm$ and scale function $s_\pm$ is
\begin{align}
	dm_\pm(x) = 2\mathrm{e}^{\pm2\sqrt{2\gamma} x}dx \quad \text{and}
	\quad
	ds_\pm(x) = \mathrm{e}^{\mp2\sqrt{2\gamma} x}dx, \label{}
\end{align}
and therefore we obtain
\begin{align}
	\frac{d}{dm_\pm}\frac{d}{ds_\pm} = \frac{1}{2}\frac{d^2}{dx^2} \pm \sqrt{2\gamma}\frac{d}{dx}, \label{}
\end{align}
which are the generator of a Brownian motion with constants drifts.

We may compute the transition density and the first hitting time densities by 
\eqref{eigen_shift}, \eqref{spectral_shift}, 
\eqref{diffusion_transition} and \eqref{eq20}.
For $\psi_{\lambda}$ and $\rho$ given in \eqref{eq123} and \eqref{eq124}, the generalized Fourier transform is
\begin{align}
	\hat{f}(\lambda) &= \int_{0}^{\infty}f(x)\frac{\psi_{-\lambda + \gamma}(x)}{k_{\pm}(x)}dm_{\pm}(x) \label{} \\
	&= \sqrt{\frac{2}{\lambda - \gamma}} \int_{0}^{\infty}f(x)
	\sin (\sqrt{2(\lambda - \gamma)} x) \mathrm{e}^{\pm 2 \sqrt{2\gamma}x}dx  \quad (f \in L^2(dm_{\pm}), \ \lambda > 0). \label{}
\end{align}
Its inverse transform is
\begin{align}
	g(x) &= \int_{0}^{\infty}g(\lambda)\frac{\psi_{-\lambda + \gamma}(x)}{k_{\pm}(x)}\rho(d(\lambda - \gamma)) \label{} \\
	&= \frac{1}{\pi}\int_{\gamma}^{\infty}g(\lambda )\sin (\sqrt{2(\lambda - \gamma)}x) \mathrm{e}^{\mp 2 \sqrt{2\gamma} x}d\lambda
	\quad (g \in L^2(\rho_{\gamma}),\  x > 0 ), \label{}
\end{align}
where $\rho_{\gamma}(d\lambda) = \rho(d(\lambda - \gamma))$.

\bibliography{report05.bbl} 

\begin{thebibliography}{10}

\bibitem{Anderson}
W.~J. Anderson.
\newblock {\em Continuous-time {M}arkov chains}.
\newblock Springer Series in Statistics: Probability and its Applications.
  Springer-Verlag, New York, 1991.
\newblock An applications-oriented approach.

\bibitem{Coddington}
E.~A. Coddington and N.~Levinson.
\newblock {\em Theory of ordinary differential equations}.
\newblock McGraw-Hill Book Company, Inc., New York-Toronto-London, 1955.

\bibitem{Ito_essentials}
K.~It\^{o}.
\newblock {\em Essentials of stochastic processes}, volume 231 of {\em
  Translations of Mathematical Monographs}.
\newblock American Mathematical Society, Providence, RI, 2006.
\newblock Translated from the 1957 Japanese original by Yuji Ito.

\bibitem{McKean:elementary}
H.~P.~McKean Jr.
\newblock Elementary solutions for certain parabolic partial differential
  equations.
\newblock {\em Trans. Amer. Math. Soc.}, 82:519--548, 1956.

\bibitem{Kallenberg}
O.~Kallenberg.
\newblock {\em Foundations of modern probability}.
\newblock Probability and its Applications (New York). Springer-Verlag, New
  York, second edition, 2002.

\bibitem{KarlinMcGregor}
S.~Karlin and J.~L. McGregor.
\newblock The differential equations of birth-and-death processes, and the
  {S}tieltjes moment problem.
\newblock {\em Trans. Amer. Math. Soc.}, 85:489--546, 1957.

\bibitem{KotaniWatanabe}
S.~Kotani and S.~Watanabe.
\newblock Kre\u{\i}n's spectral theory of strings and generalized diffusion
  processes.
\newblock In {\em Functional analysis in {M}arkov processes ({K}atata/{K}yoto,
  1981)}, volume 923 of {\em Lecture Notes in Math.}, pages 235--259. Springer,
  Berlin-New York, 1982.

\bibitem{Specialfunction}
W.~Magnus, F.~Oberhettinger, and R.~P. Soni.
\newblock {\em Formulas and theorems for the special functions of mathematical
  physics}.
\newblock Third enlarged edition. Die Grundlehren der mathematischen
  Wissenschaften, Band 52. Springer-Verlag New York, Inc., New York, 1966.

\bibitem{RevusYor}
D.~Revuz and M.~Yor.
\newblock {\em Continuous martingales and {B}rownian motion}, volume 293 of
  {\em Grundlehren der mathematischen Wissenschaften [Fundamental Principles of
  Mathematical Sciences]}.
\newblock Springer-Verlag, Berlin, third edition, 1999.

\bibitem{RogersFPP}
L.~C.~G. Rogers.
\newblock A diffusion first passage problem.
\newblock In {\em Seminar on stochastic processes, 1983 ({G}ainesville, {F}la.,
  1983)}, volume~7 of {\em Progr. Probab. Statist.}, pages 151--160.
  Birkh\"{a}user Boston, Boston, MA, 1984.

\bibitem{bernstein_functions}
R.~L. Schilling, R.~Song, and Z.~Vondra\v{c}ek.
\newblock {\em Bernstein functions}, volume~37 of {\em De Gruyter Studies in
  Mathematics}.
\newblock Walter de Gruyter \& Co., Berlin, second edition, 2012.
\newblock Theory and applications.

\bibitem{TakemuraTomisaki:htransform}
T.~Takemura and M.~Tomisaki.
\newblock {$h$} transform of one-dimensional generalized diffusion operators.
\newblock {\em Kyushu J. Math.}, 66(1):171--191, 2012.

\bibitem{preprintQSD}
K.~Yamato.
\newblock A unifying approach to non-minimal quasi-stationary distributions for
  one-dimensional diffusions.
\newblock {\em arXiv:2012.12971}, 2020.

\bibitem{report}
K.~Yamato.
\newblock Existence of laplace transforms of the spectral measures for
  one-dimensional diffusions with an exit boundary.
\newblock {\em Infinitely divisible processes and related topics (25), The
  Institute of Statistical Mathematics Cooperative Research Report}, 2021.
\newblock in preparation.

\bibitem{Yano:Excusionmeasure}
K.~Yano.
\newblock Excursion measure away from an exit boundary of one-dimensional
  diffusion processes.
\newblock {\em Publ. Res. Inst. Math. Sci.}, 42(3):837--878, 2006.

\end{thebibliography}
\bibliographystyle{plain}

\end{document}